\documentclass[hypertex]{amsart}
\usepackage{amsmath}
\usepackage{amsthm}
\usepackage{amsfonts}
\usepackage{amssymb}
\usepackage{mathrsfs}

\def\N{{\mathbb N}}

\def\Z{{\mathbb Z}}

\def\cP{\mathcal P}

\def\cT{\mathcal T}
\def\cF{\mathcal F}
\def\cO{\mathcal O}
\def\cM{\mathcal M}
\def\cP{\mathcal P}
\def\cK{\mathcal K}
\def\bR{\bold R}
\def\sX{\mathscr X}
\def\sY{\mathscr Y}

\theoremstyle{plain}
\newtheorem{theorem}{Theorem}[section]
\newtheorem{proposition}[theorem]{Proposition}
\newtheorem{corollary}[theorem]{Corollary}
\newtheorem{lemma}[theorem]{Lemma}

\theoremstyle{definition}
\newtheorem{defi}[theorem]{Definition}

\newtheorem{remark}[theorem]{Remark}

\input xy
\xyoption{all}

\usepackage{hyperref}

\begin{document}

\title{Descent properties of equivariant K-theory}
\author{Christian Serp\'{e}}
\address{University of M\"unster\\ Germany}
\email{serpe@uni-muenster.de}

\date{\today}

\begin{abstract}
  We show that equivariant $K$-theory satisfies descent with respect to the isovariant Nisnevich topology. The main step is to show that the isovariant Nisnevich topology is a regular, complete and bounded cd topology.
\end{abstract}

\maketitle



\tableofcontents

\section{Introduction}

The Zariski descent property for (non equivariant) K-theory was first
shown by Brown and Gersten in \cite{BrownGersten}. The case of Nisnevich descent first appeared in \cite{Nisnevich} and \cite{thomason_trobaugh}. In this article we want to investigate what happens in the equivariant situation. That means we consider schemes with a $G$-action for a fixed group scheme $G$ and consider the equivariant $K$-theory of the scheme. One notices first that for the topology defined by equivariant Nisnevich coverings, i.e. equivariant morphisms which are Nisnevich covers when we forget the $G$-action, the descent property is not satisfied (compare Remark \ref{counterexample}). Instead we consider the isovariant Nisnevich topology. Covers are equivariant Nisnevich covers which are isovariant, i.e. where the isotropy groups do not change (see Definition \ref{defi:isovariant}). The main result is that equivariant $K$-theory satisfies descent with respect to the isovariant Nisnevich topology. Isovariant topologies were first considered by Thomason in \cite{MR948534} for the \'etale topology. In loc. cit. Thomason showed  that equivariant $K$-theory with inverted Bott element satisfies descent with respect to the isovariant \'etale topology. To proof our main theorem we use the general formalism of cd structure of Voevodsky from \cite{Voecdstructures} and \cite{VoeNiscdhtopologies}.

\vspace{0.3cm} 

In the second section we introduce the isovariant  Nisnevich topology and show that for  schemes with scheme-theoretical geometric quotients  this topology is isomorphic to the Nisnevich topology of the quotient. This works similar to the \'etale case in \cite{MR948534}. We further show that the isovariant Nisnevich topology can be defined by a regular, complete and bounded cd structure.

In section three we recollect descent properties for topologies which are defined by regular, complete and bounded cd structures. We mainly use here the work \cite{Voecdstructures} and \cite{Blander}.

In section three we proof our main theorem that equivariant $K$-theory satisfies descent in the isovariant Nisnevich topology. 

We would like to thank Gereon Quick for useful discussions and comments.
 
\section{The isovariant Nisnevich topology}

Let $S$ be a separated noetherian scheme and $G$ a flat $S$-group schemes of finite type. All schemes are supposed to be schemes of finite type over $S$ and all fibre products are supposed to be over $S$ if not otherwise mentioned. First we recall some notations of equivariant scheme theory needed in this paper. We refer to \cite{MR948534}[\S2] for further details.

\begin{defi}
If a group scheme $G$ acts on a scheme $X$ we denote by $G_X$ the
{\em isotropy subgroup}, defined by the pullback
\begin{displaymath}
\xymatrix{
G_X\ar[r] \ar[d] & G\times X \ar[d] &(g,x)\ar@{|->}[d] \\
X\ar[r]^{\Delta_X} & X\times X & (gx,x).
}
\end{displaymath}
\end{defi}

\begin{defi}
If a group scheme $G$ acts on a scheme $X$, a morphism $\pi:X\rightarrow X/G$ is called a {\em scheme-theoretic geometric quotient} if $\pi$ is flat and represents the fppf quotient of $X$ by $G$.
\end{defi}

Important for what follows is the following proposition, which says that locally scheme-theoretic geometric quotients do exist.
\begin{proposition}[\cite{MR860673}Prop. 4.7]\label{genquotient}
Let $X$ be an reduced scheme with an $G$-action. Then there is an open dense $G$ invariant subset $U\subset X$ such that the scheme-theoretic geometric quotient of $U$ by $G$ exists.
\end{proposition}
\begin{defi}\label{defi:isovariant}
A  morphism $f:X\rightarrow Y$ between two schemes $X,Y$ with a $G$-action is
called {\em isovariant} if it is equivariant and if the following diagram is cartesian
\begin{displaymath}
\xymatrix{
G_X \ar[d]\ar[r] & G_Y \ar[d] \\
X\ar[r] &Y.
}
\end{displaymath}
\end{defi}

\begin{defi}
A family of morphisms $\{X_i\rightarrow X\}$ of $G$-schemes is called an {\em isovariant Nisnevich cover} if each morphism $X_i\rightarrow X$ is isovariant and $\{X_i\rightarrow X\}$ is a Nisnevich cover if we forget the $G$-actions. We denote by $iso-Nis$ the induced topology on various categories. For a $G$-scheme $X$ we denote by $iso-Nis/X$ the category of morphism $Y\rightarrow X$ where $Y$ is also a $G$-scheme and where the morphism is isovariant and Nisnevich if we forget the $G$-action. 
\end{defi} 

\begin{remark}
  Let $X$ be a $G$-scheme and $U\subset X$ be an open subset of
  $X$. Then the inclusion $U\hookrightarrow X$ is isovariant if and
  only if $U$ is a $G$-invariant open subset. So the small isovariant
  Zariski site is equal to the small equivariant Zariski site. But in
  general this is not true. Consider for example a scheme $X$ with the
  trivial action of the group $\Z/2$ and $X\coprod X$ with the $\Z/2$
  action that switches the components. Then the obvious morphism
  $X\coprod X\rightarrow X$ is equivariant and a Zariski cover, but
  the morphism is not  a cover in the isovariant topology. This example already reveals the reason why we have to consider the isovariant topology for equivariant K-theory. We postpone the reader to Remark \ref{counterexample} for a more detailed explanation. 
\end{remark}

For a scheme $Y$ we denote by $Nis/Y$ the category of morphisms $U\rightarrow T$ that are Nisnevich.

\begin{proposition}\label{sitesequi}
  Let $X$ be a $G$-scheme with a scheme-theoretical geometric quotient $\pi:X\rightarrow X/G$. Then we have the following equivalences of sites:
\begin{displaymath}
  Nis/(X/G)\rightarrow iso-Nis/X \text{, } U\hookrightarrow (X/G) \mapsto U\times_{X/G}X\rightarrow X
\end{displaymath}
Here $G$ acts on $U\times_{X/G}X$ via the second factor.
\end{proposition}

\begin{proof}
  To proof the statement we construct similar as in \cite{MR860673} the inverse of the above functor. So let $W\xrightarrow{f} X$ be an isovariant Nisnevich map. Because $f$ is isovariant the diagram
\begin{displaymath}
  \xymatrix{
    W\ar[r]^f \ar[d] & X \ar[d] \\
    W/G \ar[r]^{f'} & X/G
    } 
\end{displaymath}
is cartesian. Now by assumption $X/G$ is a scheme, $W/G$ is an algebraic spaces, and by descent the morphism $f'$ is \'etale. Therefore by \cite{KnAlgSpc}[Cor. 6.17] $W/G$ is a scheme and by the fact that $W\rightarrow X$ is Nisnevich and isovariant it follows that the \'etale morphism $W/G\rightarrow X/G$ is also Nisnevich.
\end{proof}

Now we consider the category $\,^GSm/S$ of $G$-scheme, smooth and of finite type over $S$, with the induced isovariant Nisnevich topology and we want to show, that this topology is in fact a cd topology in the sense of \cite{Voecdstructures}. For the basic definition and properties of cd structures we refer to loc. cit..

\begin{remark}
  Instead of $\,^GSm/S$ we could also consider $\,^GSch/S$ the category of all separated $G$-scheme of finite type over S. All the results in this section would also be valid for  $\,^GSch/S$. 
\end{remark}  

\begin{defi}
  A {\em distinguished square} in $\,^GSm/S$ is a cartesian diagram of the form
\begin{equation}\label{distingsquare}
  \xymatrix{
    U\times_X V \ar[r] \ar[d] & V \ar[d]^p \\
    U \ar[r]^i & X
    }
\end{equation}
such that $p$ is an isovariant \'etale morphism, $i$ is an invariant open embedding and $p^{-1}(X-U)\rightarrow X-U$ is an isomorphism. The set of all distinguished squares in $(\,^GSm/S)_{iso-Nis}$ defines a cd topology on $\,^GSm/S$, which we call the {\em isovariant Nisnevich cd structure}.
\end{defi}

The next proposition shows that the topology defined by this cd structure is the same as $iso-Nis$.

\begin{proposition}\label{cdtopologytheorem}
  The isovariant Nisnevich topology on $\,^GSm/S$ coincides with the
  topology defined by the isovariant Nisnevich cd structure.
\end{proposition}

For the proof we need the following notation and a lemma. 
\begin{defi}
  An {\em equivariant splitting sequence}  for an equivariant morphism $p:X\rightarrow Y$ between two $G$-schemes $X,Y$ is a sequence of invariant closed subschemes
$$ \emptyset=Z_{n+1}\hookrightarrow Z_n\hookrightarrow \dots \hookrightarrow Z_1\hookrightarrow Z_0=Y$$
such that for all $i=1,\dots, n$ the morphism
$$p|_{p^{-1}(Z_i-Z_{i+1})}:p^{-1}(Z_i-Z_{i+1})\rightarrow Z_i-Z_{i+1}$$
has a equivariant section.
\end{defi}

\begin{lemma}\label{exsplittinseq}
  Let $p:\tilde X\rightarrow X$ be an isovariant Nisnevich covering map. Then $p$ has an equivariant splitting sequence.  
\end{lemma}

\begin{proof}
  By \ref{genquotient} there is an open $G$-invariant $U\subset X$ such that $U$ has a scheme-theoretical geometric quotient $U\xrightarrow{\pi} U/G$. The morphism $\tilde X_U:=\tilde X\times_XU\rightarrow U$ corresponds by \ref{sitesequi} to a usual Nisnevich covering $\tilde X_U/G\rightarrow U/G$. There is an open subset $V'\subset U/G$ such that $(\tilde X_U/G)\times_{U/G} V'\rightarrow V'$ has a splitting. Pulling $V'$ back to $X$ gives an open $G$-invariant $V:=V'\times_{U/G}U\subset U\subset X$ with an equivariant splitting of the morphism
$$ \tilde X/G\times_{U/G}V'\times_{U/G}U\cong\tilde X\times_X V \rightarrow V$$
We define $Z_1:=X-V$ with the reduced scheme structure. Now $\tilde X_{Z_1}\rightarrow Z_1$ is again Nisnevich and we can repeat the procedure. By noetherian induction this completes the proof. 
\end{proof}

The following is the equivariant version of the proof of   \cite{MR1813224}[Proposition 3.1.4] resp. \cite{VoeNiscdhtopologies}[Proposition 2.16]
\begin{proof}[Proof of \ref{cdtopologytheorem}.]
Because a distinguished square of the isovariant Nisnevich cd structure is an isovariant Nisnevich cover we only have to show that an isovariant Nisnevich cover has a refinement by a covering of the topology defined by the isovariant Nisnevich cd structure.

Without loss of generality we consider a isovariant Nisnevich covering of the simple form $p:\tilde{X}\rightarrow X$. By lemma  \ref{exsplittinseq} there is an equivariant splitting sequence 
\begin{displaymath}
  \emptyset=Z_{n+1}\hookrightarrow Z_n\hookrightarrow \dots \hookrightarrow Z_1\hookrightarrow Z_0=X
\end{displaymath}
for the map $p$. We show that there is a refinement by induction on the minimal length $n$ of the splitting sequence. If $n=0$ we have an equivariant splitting of $p:\tilde{X}\rightarrow X$ and therefore the covering coming from the distinguished square
\begin{displaymath}
  \xymatrix{
    \tilde{X}\ar[d]\ar[r] & \tilde{X} \ar[d] \\
    X \ar[r]               & X
}
\end{displaymath}
is a refinement of $\tilde{X}\rightarrow X$. Now let $n>0$. We consider an equivariant section $s$ of $f_n:\tilde{X}\times_XZ_n\rightarrow Z_n$ and because $f_n$ is \'etale we can decompose equivariantly
\begin{displaymath}
  \tilde{X}\times_X Z_n= s(Z_n)\coprod W.
\end{displaymath}
We set $V:=\tilde{X}-W$ and $U:=X-Z_n$ and get the elementary square
\begin{displaymath}
  \xymatrix{
    U\times_X V \ar[r]\ar[d] & V \ar[d] \\
    U \ar[r]                 &X
}
\end{displaymath}
The pullback of $\tilde{X}\rightarrow X$ along $V\rightarrow X$ has a section and along $U\rightarrow X$ has a splitting sequence of length $n-1$. So by induction the claim follows. 
\end{proof}

Next we want to show some properties of the isovariant Nisnevich cd structure.

\begin{proposition}
  The isovariant cd structure is complete and regular.
\end{proposition}

\begin{proof}
  Because the pullback of a distinguished square is again distinguished the isovariant Nisnevich cd structure is complete (see \cite{Voecdstructures}[Lemma 2.5]). Now we show regularity. By \cite{Voecdstructures}[Lemma 2.11] we only have to check the following. If
\begin{displaymath}
  \xymatrix{
    B \ar[d]\ar[r] & Y \ar[d] \\
    A \ar[r]       & X 
}
\end{displaymath}
is a distinguished square then the cartesion square
\begin{displaymath}
  \xymatrix{
    B \ar[r] \ar[d]^{\Delta_{B/A}} & Y \ar[d]^{\Delta_{Y/X}} \\
    B\times_A B \ar[r] & Y\times_X Y
}
\end{displaymath}
is also distinguished.
By \cite{ega43}[17.42] $Y\rightarrow Y\times_X Y$ is an open immersion because $Y\rightarrow X$ is unramified. Because $Y\rightarrow Y\times_X Y$ is also $G$-invariant follows that $Y\rightarrow  Y\times_X Y$ is isovariant. Because $A\rightarrow X$ is an open immersion $B\times_A B\rightarrow Y\times_X Y$ is an open immersion. As before the isovariance of $B\times_A B\rightarrow Y\times_X Y$ follows. Furthermore one shows that 
$$ Y\times_X Y=(B\times_A B) \cup Y.$$
\end{proof}

Our next goal is to show that the isovariant cd structure is bounded. For that we first introduce the isovariant natural density structure.

\begin{defi}
  An {\em increasing sequence of length n} in a topological space $X$ is a sequence of points $x_0,x_1,\dots,x_n$ of $X$ such that $x_i\neq x_{i+1}$ and $x_i\in\overline{\{x_{i+1}\}}$. For a $G$-scheme $X$ we define $D^G_d(X)$ as the class of all $G$-invariant open embeddings $j:U\hookrightarrow X$ such that for all $z\in X-U$ there is an increasing sequence $z=x_0,x_1,\dots, x_d$ in $X$ of length d.
\end{defi}

\begin{remark}
  For a $G$-scheme $X$ an open subscheme $U\hookrightarrow X$ is in $D^G_d(X)$ if and only if $U\hookrightarrow X$ is $G$-invariant and if $U\hookrightarrow X$ is in $D_d(X)$ when we forget the $G$-action.
\end{remark}
Obviously we have the following lemma.
\begin{lemma}
  The classes $D^G_d$ define a density structure on $\,^GSm/S$ in the sense of  \cite{Voecdstructures}[Definition 2.20]. This density structure is locally finite on finite dimensional $G$-schemes and the dimension of a scheme with respect to this density structure is less or equal to the usual dimension of the scheme.
\end{lemma}

\begin{proposition}
  The isovariant cd structure on $\,^GSm/S$ is bounded with respect to the density structure $D^G$.
\end{proposition}

Even though the proof is similar to the non equivariant version we
don't shy away from doing the proof in some details, especially
because in the proof of \cite{VoeNiscdhtopologies}[Proposition 2.10]
there are some misprints.

\begin{proof}
  Let 
  \begin{displaymath}
    \xymatrix{
      W\ar[r]^{j_Y} \ar[d] & Y \ar[d]^{p} \\
      U \ar[r]^{j} & X
      }
  \end{displaymath}
  be a distinguished square in the isovariant Nisnevich cd structure on $\,^GSm/S$ and let $W_0\in D^G_{d-1}(W), U_0\in D^G_d(U), Y_0\in D^G_d(Y)$. Then we have to find another distinguished square
\begin{displaymath}
    \xymatrix{
      W'\ar[r] \ar[d] & Y' \ar[d]^{p} \\
      U' \ar[r] & X'
      }
\end{displaymath}
and a morphism of distinguished squares
\begin{displaymath}
  \xymatrix{
            & W'\ar[rr] \ar[dd]|\hole \ar[ld]&     & Y'\ar[dd]\ar[ld] \\
       W \ar[rr] \ar[dd] & & Y\ar[dd] &            \\
            & U' \ar[rr]|\hole \ar[ld]      &     & X' \ar[ld]      \\
       U \ar[rr]       & & X
       }  
\end{displaymath}
such that $W'\rightarrow W$ factors through $W_0\rightarrow W$, $Y'\rightarrow Y$ factors through $Y_0\rightarrow Y$, $U'\rightarrow U$ factors through $U_0\rightarrow U$, and $(X'\rightarrow X)\in D^G_d(X)$. The morphism $U\coprod Y\xrightarrow{j\coprod p} X$ is open, surjective, and has zero dimensional fibers. Therefore 
$(j\coprod p) (U_0\coprod Y_0) =:X_0$ is an open $G$-invariant subset of $X$ and $X_0\in D^G_d(X)$. By base change with the morphism $X_0\rightarrow X$ we can assume that $U_0=U$ and $Y_0=Y$ (see \cite{VoeNiscdhtopologies} Lemma 2.4).

Now we take $W':=W_0$, $U':=U$, $Y':=Y-\overline{j_Y(W-W_0)}$, and 
$$X':=X-[(X-U)\cap \overline{(p\circ j_Y(W-W_0))}].$$

First one notices that $Y'$ resp. $X'$ is indeed an open $G$-invariant subset of $Y$ resp. $X$ and that this defines a distinguished square such that the required factorization properties are satisfied. The only thing left to check is whether $X'\in D^G_d(X)$. For that let $x\in X-X'$. Because $(X-U)\cap p\circ j_V(W-W_0)=\emptyset$ and \cite{VoeNiscdhtopologies}[Lemma 2.7] there is an $x_1\in p\circ j_V(W-W_0)$ with $x\in\overline{\{x_1\}}$ and $x\not= x_1$. Now because $p\circ j_V$ has zero dimensional fibers and $W_0\in D^G_{d-1}(W)$ we find an increasing sequence $x_1,x_2,\dots x_d$ in $X$. Joining $x_0:=x$ we see that $X'\in D^G_d(X)$. 
\end{proof}

\begin{corollary}\label{finitedimension} 
  The cohomological dimension of the isovariant Nisnevich topology of a finite dimensional scheme is less or equal to the dimension of the scheme.
\end{corollary}
\begin{proof}
  See \cite{Voecdstructures}[Theorem 2.27]
\end{proof}


\begin{remark}
Voevodsky defined in \cite{DeligneVoe1} an equivariant topology for a
finite flat group scheme $G$. With Lemma \ref{exsplittinseq} and
\cite{DeligneVoe1}[Lemma 2] one can show that in this case the equivariant topology of \cite{DeligneVoe1} coincides with the isovariant Nisnevich topology. 
\end{remark}

\section{Descent for cd topologies}

In this section we recall shortly what is known about the descent property for complete, regular, and bounded cd structures. So let $P$ a complete, regular, and bounded cd structure on a category $\cT$. We denote by $t_p$ the topology defined by $P$. 

Let $PreShv(\cT)$ be the category of set valued presheaves on $\cT$ and $Shv_{t_p}(\cT)\subset PreShv(\cT)$ the full subcategory of sheaves on $\cT$ in the topology $t_p$. We have the natural inclusion
\begin{displaymath}
  i:Shv_{t_p}(\cT) \rightarrow PreShv(\cT)
\end{displaymath}
which has as right adjoint the sheafification functor
\begin{displaymath}
  a_{t_p}: PreShv(\cT)\rightarrow Shv_{t_p}(\cT).
\end{displaymath}

A presheaf $\cF$ is basically by definition a sheaf if and only if the natural morphism
\begin{displaymath}
  \cF\rightarrow (i\circ a_{t_p})(\cF)
\end{displaymath}
is an isomorphism. On the other hand, because the cd structure $P$ is complete, we know that a presheaf $\cF$ is a sheaf if and only if $\cF(\emptyset)=*$ and if for each distinguished square
\begin{displaymath}
  \xymatrix{
    B\ar[r]\ar[d] & Y\ar[d] \\
    A\ar[r]       & X 
    }
\end{displaymath}
the square
\begin{displaymath}
  \xymatrix{
    \cF(X)\ar[r] \ar[d] & \cF(Y) \ar[d] \\
    \cF(A)\ar[r]        & \cF(B) 
    }
\end{displaymath}
is cartesian \cite{Voecdstructures}[Lemma 2.9]. Of course the same remains true if we replace the category of presheaves by simplicial presheaves $\Delta^{op}PreShv(\cT)$ and the category of sheaves by the category of simplicial sheaves $\Delta^{op}Shv_{t_p}(\cT)$. The goal of this section is to show that a derived version of this statement is also true.

We recall the following definitions.
\begin{defi}
  A map $f:\sX\rightarrow \sY $ of simplicial presheaves is called a  {\em global (or section wise) weak equivalence} if for all $U\in\cT$ the map
\begin{displaymath}
  f(U):\sX(U)\rightarrow \sY(U)
\end{displaymath}
is a weak equivalence of simplicial sets. 

  A map $f:\sX\rightarrow \sY $ of simplicial presheaves is called a  {\em local weak equivalence} (with respect to $t_p$) if 
\begin{displaymath}
  f_*:a_{t_p}\pi_0(\sX)\rightarrow a_{t_p}\pi_0(\sY)
\end{displaymath}
is an isomorphism and if for all $U\in\cT$ and all $x\in\sX(U)$
\begin{displaymath}
  f_*:a_{t_p}\pi_n(\sX,x)\rightarrow a_{t_p}\pi_n(\sY,f(x))
\end{displaymath}
is an isomorphism
\end{defi}

Now we consider the localization of $\Delta^{op}PreShv(\cT)$ with respect to global weak equivalences and local weak equivalences and also the localization of $\Delta^{op}Shv_{t_p}(\cT)$ with respect to local weak equivalences:
\begin{eqnarray*}
\Delta^{op}PreShv(\cT)&\rightarrow& Ho((\Delta^{op}PreShv(\cT))_g)\\
\Delta^{op}PreShv(\cT)&\rightarrow& Ho((\Delta^{op}PreShv(\cT))_l)\\
   \Delta^{op}Shv_{t_p}(\cT)&\rightarrow& Ho((\Delta^{op}Shv_{t_p}(\cT))_l)
\end{eqnarray*}

  If $f:\sX\rightarrow \sY$ is a global weak equivalence of simplicial presheaves then $a_{t_p}(f): a_{t_p}(\sX)\rightarrow a_{t_p}(\sY)$ is local weak equivalence of simplicial sheaves.
Therefore $a_{t_p}$ induces a functor between the homotopy categories making the following diagram commutative.
\begin{displaymath}
  \xymatrix{
    \Delta^{op}PreShv(\cT)\ar[r]^{a_{t_p}}\ar[d] & \Delta^{op}Shv_{t_p}(\cT)\ar[d] \\
    Ho((\Delta^{op}PreShv(\cT))_g) \ar[r] &Ho((\Delta^{op}Shv_{t_p}(\cT))_l) 
} 
\end{displaymath}

Now we consider the right derived functor of
\begin{displaymath}
  i:\Delta^{op}Shv_{t_p}(\cT)\rightarrow \Delta^{op}PreShv(\cT)\rightarrow Ho((\Delta^{op}PreShv(\cT))_g).
\end{displaymath}

So we have the diagram
{\small
\begin{displaymath}
  \xymatrix{
    \Delta^{op}PreShv(\cT) \ar[r]\ar[d] &
    \Delta^{op}Shv_{t_p}(\cT)\ar[r]\ar[d] &
    \Delta^{op}PreShv(\cT)\ar[d] \\
    Ho((\Delta^{op}PreShv(\cT))_g)\ar[r]^{a_{t_p}} &
    Ho((\Delta^{op}Shv_{t_p}(\cT))_l)\ar[r]^{\bR i} &
    Ho((\Delta^{op}PreShv(\cT))_g)
    }  
\end{displaymath}}

\begin{defi}
  A simplicial presheaf $\cF$ has {\em descent with respect to $t_p$}
  if the natural map
  \begin{displaymath}
    \cF\rightarrow \bR i \circ a_{t_p}(\cF)
  \end{displaymath}
is an isomorphism in $Ho((\Delta^{op}PreShv(\cT))_g)$.
\end{defi}

\begin{defi}
  A simplicial presheaf $\sX$ is called {\em flasque} with respect to the cd structure $P$ if $\sX(\emptyset)$ is contractible and if for  each distinguished square in $P$ 
\begin{equation}\label{distsquare}
    \xymatrix{
      B\ar[r]\ar[d] & Y\ar[d]\\
      A\ar[r]       & X      
      }
\end{equation}
  
the diagram
\begin{displaymath}
  \xymatrix{
    \sX(X)\ar[d]\ar[r] & \sX(Y)\ar[d] \\
    \sX(A)\ar[r]       & \sX(B) 
    }
\end{displaymath}
is homotopy cartesian in the category of simplicial sets. 
\end{defi}

The main result of this section is the following.
\begin{proposition}\label{descenttheorem}
  A simplicial presheaf $\sX$ has decent with respect to $t_p$ if and
  only if $\sX$ is flasque with respect to $P$.
\end{proposition}
Before we proof the proposition we show the following lemma.
\begin{lemma}\label{fibrantflasque}
  Let $\sX'$ be a fibrant sheaf in the local projective model structure of \cite{Blander} on $\Delta^{op}Shv_{t_p}$. Then $\sX'$ is flasque.
\end{lemma}

\begin{proof}
 Because $\sX'$ is a sheaf we see that $\sX'(\emptyset)=*$ and  that for each distinguished square (\ref{distsquare}) the diagram
\begin{equation}\label{homotopycatesian}
  \xymatrix{
    \sX'(X)\ar[d]\ar[r] & \sX'(Y)\ar[d] \\
    \sX'(A)\ar[r]       & \sX'(B) 
    }
\end{equation}
is cartesian in the category of simplicial sets. Now since a fibrant object in the local projective model structure on $\Delta^{op}Shv_{t_p}(\cT)$ is also fibrant as an object in $\Delta^{op}PreShv(\cT)$ with the global projective model structure we see that for any $T\in\cT$ the $\sX'(T)$ is a fibrant simplicial set. Therefore the diagram (\ref{homotopycatesian}) is also homotopy cartesian.
\end{proof}
\begin{proof}[Proof of Proposition \ref{descenttheorem}]
  We show first the "only if" part. We consider the local projective model structure of \cite{Blander} on $\Delta^{op}Shv_{t_p}(\cT)$. We choose a fibrant resolution
\begin{displaymath}
  a_{t_p}(\sX)\rightarrow \sX'
\end{displaymath}
in this model structure and we know that
\begin{displaymath}
  \bR i a_{t_p}(\sX)\simeq  i(\sX')
\end{displaymath}
in $Ho((\Delta^{op}PreShv(\cT))_g)$. Now the claim follows from lemma \ref{fibrantflasque}.

Now we show the "if" part. So we assume that $\sX$ is flasque. Let as before
\begin{displaymath}
  a_{t_p}(\sX)\rightarrow \sX'
\end{displaymath}
be a fibrant resolution in the local projective model structure. First by Lemma \ref{fibrantflasque} $\sX'$ is flasque and $\sX \rightarrow a_{t_p}(\sX)\rightarrow\sX'$ is a $t_p$-local isomorphism. Therefore by \cite{Voecdstructures}[Lemma 3.5]  this resolution is a global weak equivalence and the claim follows. 
\end{proof}

Next we want to have an analogous result of Proposition \ref{descenttheorem} for presheaves of spectra. For that we use the model structure on presheaves of spectra on a Grothendieck topology of \cite{Jardinestalbehomosimpl}. A morphism of presheaves of spectra $f:\sX\rightarrow \sY$ is called a stable equivalence if for all $n\in\N$ the morphism
\begin{displaymath}
  \tilde{\pi}^s_n(\sX)\rightarrow \tilde{\pi}^s_n(\sY)
\end{displaymath}
is an isomorphism of sheaves on $\cT$. Here we denote by $\tilde{\pi}^s_n $ the sheafification of the stable $n$-th homotopy groups. As in \cite{JardineMotivicSpectra} one can deduce from \ref{descenttheorem} the following result.
\begin{proposition}
  Let $\sX$ be a presheaf of spectra on $\cT$ and $\sX\rightarrow \sX'$ be a fibrant resolution in the stable model structure. Then the morphism $\sX(U)\rightarrow \sX'(U)$ is a stable equivalence of spectra for all $U\in\cT$ if and only if $\sX$ is stably flasque, i.e. for each distinguished square (\ref{distsquare}) the diagram 
\begin{displaymath}
  \xymatrix{
    \sX(X)\ar[d]\ar[r] & \sX(Y)\ar[d] \\
    \sX(A)\ar[r]       & \sX(B) 
    }
\end{displaymath}
is homotopy cartesian in the stable model structure on the category of spectra(cf. \cite{MR513569}). 
\end{proposition}

\section{Equivariant algebraic K-theory}

Let $k$ be a field and $G$ be an algebraic group over $k$. We denote by $\,^GSm/k$ the category of smooth schemes over $k$ together  with an action of the group $G$. Morphisms are equivariant morphism. For $X\in\,^GSm/k$ a coherent $G$-module is a coherent $\cO_X$-module $\cF$ together with an action of $G$ on $\cF$ which is compatible with the action of $G$ on $X$. For what that means in details we refer to \cite{MR921490}[section 1.2]. We denote by $\cM_{G,X}$ the abelian category of coherent $G$-modules and by $\cP_{G,X}$ the exact category of those coherent $G$-modules which are locally free $\cO_X$-modules. Furthermore the $K$-theory spectra of these categories are respectively denoted by $G(G,X)$ and $K(G,X)$.  $K(G,X)$ is contravariant with respect to $G$-morphisms, and contravariant with respect to morphisms in $G$. $G(G,X)$ is contravariant with respect to flat $G$-maps in $X$ and homomorphisms in $G$.

We denote by $\cK_G$ the equivariant $K$-theory presheaves given by
\begin{displaymath}
  \cK_G: \,^GSm/k\rightarrow Spt; X\mapsto K(G,X)
\end{displaymath}

\begin{remark}\label{counterexample}
  We want to see that equivariant K-theory does not have descent in the topology where the coverings are just equivariant Nisnevich coverings. For that we consider the following examples. Let $G=\Z/2$ and $k$ be a field with $char(k)\not= 2$ and with trivial $\Z/2$ action. Further let $Spec(k)\coprod Spec(k)$ be the $G$-scheme where the nontrivial element of $\Z/2$  switches the two components. The morphism

\begin{displaymath}
  Spec(k)\coprod Spec(k) \rightarrow Spec(k)
\end{displaymath}
is compatible with he $\Z/2$ action, i.e. it is an equivariant Nisnevich covering. If equivariant K-theory satisfied descent with respect to the equivariant Nisnevich topology, then the diagram
{\tiny
\begin{equation}\label{counterexamplediagram}
  \xymatrix{
    K(G,Spec(k))\ar[r]\ar[d]   &   K(G,Spec(k)\coprod Spec(k))\ar[d]\\
    K(G, Spec(k)\coprod Spec(k)) \ar[r]   & K(G,(Spec(k)\coprod Spec(k))\times_{Spec(k)}(Spec(k)\coprod Spec(k)))
     }
\end{equation}}

would be homotopy cartesian. Because $G$ acts freely on $Spec(k)\coprod Spec(k)$ we can identify
{\tiny
\begin{eqnarray*}
  K_*(G, Spec(k)\coprod Spec(k)) & \simeq & K_*(Spec(k)) \\
  K_*(G, (Spec(k)\coprod Spec(k))\times_{Spec(k)}(Spec(k)\coprod Spec(k)))   & \simeq & K_*(Spec(k)\coprod Spec(k)).
\end{eqnarray*}
}

Using this we consider now the long exact sequence of homotopy groups corresponding to (\ref{counterexamplediagram}) and get:

{\tiny 
\begin{displaymath}
  \xymatrix{
    K_1(Spec(k)\coprod Spec(k)) \ar[r] \ar[d]^{\sim} & K_0(G, Spec(k)) \ar[r] \ar[d]^{\sim} & K_0(Spec(k))^2 \ar[r] \ar[d]^{\sim} & K_0(Spec(k)\coprod Spec(k)) \ar[r] \ar[d]^{\sim} & 0  \\
k^*\oplus k^* \ar[r] & \Z\oplus\Z \ar[r] & \Z\oplus\Z \ar[r] & \Z\oplus\Z \ar[r] & 0
}
\end{displaymath}
}

But this leads to a contradiction.
\end{remark}

The main theorem is the following.
\begin{theorem}\label{maintheorem}
The equivariant K-theory presheaf of spectra $\cK_G$ on $\,^GSm/k$ has the descent property with respect to the equivariant isovariant Nisnevich topology.
\end{theorem}
\begin{proof}
By proposition \ref{descenttheorem} we have to show that for each distinguished square
\begin{equation}\label{distingueshed}
  \xymatrix{
    V\ar[r] \ar[d] & Y\ar[d]^{p} \\
    U\ar[r]^{i}    & X
    }
\end{equation}
in the isovariant Nisnevich site, the square of spectra
\begin{equation}\label{homotopysq}
  \xymatrix{
    \cK_G(Y)\ar[r] \ar[d]  & \cK_G(Y)\ar[d]  \\
    \cK_G(U) \ar[r]    & \cK_G(V)
    }
\end{equation}
is homotopy cartesian in the category of spectra. Because $X,U,Y$ and $V$ are smooth over a field and by \cite{MR921490}[Theorem 5.7] we can instead consider the digram of spectra
\begin{equation}\label{Ghomsquare}
  \xymatrix{
    G(G,X)\ar[r]\ar[d] & G(G,Y) \ar[d]\\
    G(G,U)\ar[r]       & G(G,V).
    } 
\end{equation}
Now by the localization theorem for equivariant $G$-theory \cite{MR921490}[Theorem 2.7] we can identify the homotopy fibers of the vertical morphisms of  (\ref{Ghomsquare}):
\begin{eqnarray*}
  hofib(G(G,X) \rightarrow G(G,U)) & \simeq & G(G,X-U)\\
  hofib(G(G,Y)  \rightarrow  G(G,V)) & \simeq & G(G,Y-V)
\end{eqnarray*}

By the fact that the square (\ref{distingueshed}) is distinguished the morphism
\begin{displaymath}
  p|_{Y-V}:Y-V\rightarrow X-U
\end{displaymath}
is an isomorphism. So the induced morphism between the homotopy fibers $G(G,X-U)\simeq G(G,Y-V)$ is a stably weak equivalence. Therefore the square (\ref{homotopysq}) is homotopy cartesian.
\end{proof}

\begin{remark}
The same arguments work if we replace $Spec(k)$ by a  regular scheme $S$. The reason why we restrict ourselves to the regular case is mainly that we don't have an equivariant version of the general localization sequence from \cite{thomason_trobaugh}.  
\end{remark}

\begin{corollary}\label{local-global-SS}
  Let $X$ be a smooth scheme of finite type over a field. Then there is a strongly convergent spectral sequence
\begin{displaymath}
  E_2^{p,q}=H^p_{isovar}(X, a_{isovar}\cK_q)\rightarrow K_{-p-q}(G,X)
\end{displaymath}
   called the descent spectral sequence.
\end{corollary}
\begin{proof}
  The existence follows in the usual way one gets the descent spectral sequence (cf. for example \cite{MR1437604} section 6.1) and Theorem \ref{maintheorem}. The convergence follows from Proposition \ref{finitedimension}.
\end{proof}
\begin{remark}
  From Theorem \ref{maintheorem} it follows that equivariant $K$-theory also has the descent property with respect to the isovariant Zariski topology and that the analog of Corollary \ref{local-global-SS} for the isovariant Zariski topology is true.
\end{remark}


\bibliographystyle{alpha}
\def\cprime{$'$}

\end{document}